\documentclass[a4paper,10pt]{amsart}

\usepackage{amsmath,amssymb,amsthm,a4wide,graphicx,tikz}
\usetikzlibrary{shapes}

\newcommand{\cF}{\mathcal{F}}
\newcommand{\bI}{\mathbb{I}}
\newcommand{\bL}{\mathbb{L}}
\newcommand{\gopt}{g_{\text{opt}}}

\newtheorem*{thm}{Theorem}
\newtheorem*{lemma}{Lemma}

\begin{document}
\title[]{Optimal Jittered Sampling for\\ two points in the unit square}

\author[]{Florian Pausinger}
\address[Florian Pausinger]{Zentrum Mathematik (M10), TU Munich, Germany}
\email{florian.pausinger@ma.tum.de}

\author[]{Manas Rachh}
\address[Manas Rachh]{Applied Mathematics Program, Yale University, New Haven, CT 06510, USA}
\email{manas.rachh@yale.edu}

\author[]{Stefan Steinerberger}
\address[Stefan Steinerberger]{Department of Mathematics, Yale University, New Haven, CT 06510, USA}
\email{stefan.steinerberger@yale.edu}

\begin{abstract} Jittered Sampling is a refinement of the classical Monte Carlo 
sampling method. 
Instead of picking $n$ points randomly from $[0,1]^2$, one partitions the unit 
square into $n$ regions of equal measure and then chooses a point randomly from
each partition. 
Currently, no good rules for how to partition the space are available.
In this paper, we present a solution for the  special case of 
subdividing the unit square by a decreasing function into two regions so 
as to minimize the expected squared $\mathcal{L}_2-$discrepancy. 
The optimal partitions 
are given by a \textit{highly} nonlinear integral equation for which we
determine an approximate solution. In particular, there is a break of symmetry
and the optimal partition is \textit{not} into two sets of equal measure.
We hope this stimulates further interest in the construction of good partitions. 
\end{abstract}

\maketitle
\vspace{-15pt}
\section{Introduction and Statement of Result}
\subsection{Jittered Sampling.}
Jittered Sampling is a mixture of classical Monte Carlo and 
sampling along grid-type structures: a standard approach is to partition 
$[0,1]^2$ into $m^2$ axis aligned cubes of equal measure and placing a random 
point inside each of the $N = m^2$ cubes, see Figure \ref{fig:def}. 
This idea seems to date back to a paper of Bellhouse \cite{bellhouse} from 1981
and makes a reappearance 
in computer graphics in a 1984 paper of Cook, Porter \& Carpenter \cite{port} 
by the name of \textit{Jittered Sampling} (see also \cite{dob}). 
Bounds on the discrepancy are due to Beck \cite{beck87}; we also refer to the 
book Beck \& Chen \cite{beckchen},  
the exposition in Chazelle \cite{chazelle}, and the
recent quantified version of the first and third author \cite{Jittered}. 
Deterministic lower bounds are derived in Chen \& Travaglini \cite{chen}.
\begin{center}
\begin{figure}[h!]
\centering
\begin{tikzpicture}[scale=0.65]
\draw[step=1cm,gray,very thin] (0,0) grid (5,5);
\draw[step=1cm,gray,very thin] (8,0) grid (13,5);
%left point set
\node at (0,0) {$\bullet$}; 
\node at (0,1) {$\bullet$}; 
\node at (0,2) {$\bullet$}; 
\node at (0,3) {$\bullet$}; 
\node at (0,4) {$\bullet$}; 

\node at (1,0) {$\bullet$}; 
\node at (1,1) {$\bullet$}; 
\node at (1,2) {$\bullet$}; 
\node at (1,3) {$\bullet$}; 
\node at (1,4) {$\bullet$}; 

\node at (2,0) {$\bullet$}; 
\node at (2,1) {$\bullet$}; 
\node at (2,2) {$\bullet$}; 
\node at (2,3) {$\bullet$}; 
\node at (2,4) {$\bullet$}; 

\node at (3,0) {$\bullet$}; 
\node at (3,1) {$\bullet$}; 
\node at (3,2) {$\bullet$}; 
\node at (3,3) {$\bullet$}; 
\node at (3,4) {$\bullet$}; 

\node at (4,0) {$\bullet$}; 
\node at (4,1) {$\bullet$}; 
\node at (4,2) {$\bullet$}; 
\node at (4,3) {$\bullet$}; 
\node at (4,4) {$\bullet$}; 

%right point set
\node at (8.31,0.23) {$\bullet$}; 
\node at (8.7,1.5) {$\bullet$}; 
\node at (8.2,2.7) {$\bullet$}; 
\node at (8.3,3.34) {$\bullet$}; 
\node at (8.9,4.29) {$\bullet$}; 

\node at (9.3,0.74) {$\bullet$}; 
\node at (9.55,1.56) {$\bullet$}; 
\node at (9.1,2.83) {$\bullet$}; 
\node at (9.65,3.42) {$\bullet$}; 
\node at (9.45,4.43) {$\bullet$}; 

\node at (10.34,0.41) {$\bullet$}; 
\node at (10.67,1.66) {$\bullet$}; 
\node at (10.52,2.42) {$\bullet$}; 
\node at (10.67,3.37) {$\bullet$}; 
\node at (10.52,4.12) {$\bullet$}; 
\node at (11.5,0.82) {$\bullet$}; 
\node at (11.7,1.58) {$\bullet$}; 
\node at (11.3,2.21) {$\bullet$}; 
\node at (11.9,3.14) {$\bullet$}; 
\node at (11.7,4.72) {$\bullet$}; 
\node at (12.23,0.63) {$\bullet$}; 
\node at (12.56,1.61) {$\bullet$}; 
\node at (12.89,2.25) {$\bullet$}; 
\node at (12.17,3.57) {$\bullet$}; 
\node at (12.67,4.19) {$\bullet$}; 
\end{tikzpicture}
\caption{The regular grid and a point set obtained by Jittered Sampling.} \label{fig:def}
\end{figure}
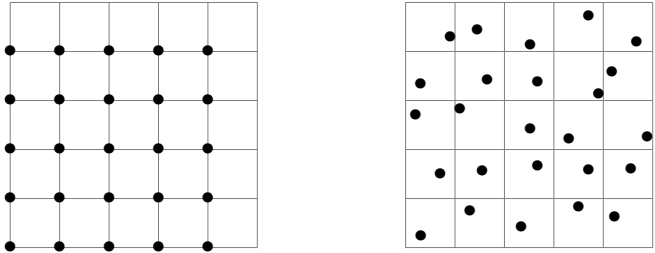
\end{center}
\vspace{-5pt}
\begin{quote}
\textbf{Main Problem.} For any $N \in \mathbb{N}$, $d \geq 2$, 
which partition of the unit cube $[0,1]^d$ into $N$ sets gives, in expectation, 
the best result for the Jittered Sampling construction? 
What partition of space should be used?
\end{quote}
\vspace{10pt}
We will study this question restricted to $d=2$. Clearly, `best' requires a 
quantitative measure of equidistribution:  
we will work with the squared $\mathcal{L}_2-$discrepancy, which,
for a given set $P = \left\{p_1, \dots, p_N\right\}$ of $N$ points in 
$[0,1]^2$ is given by
$$\mathcal{L}^2_2(P) :=  \int_{[0,1]^2}
{ \left| \frac{ \#P \cap ([0, x] \times [0,y]) }{\#P} -  x y   \right|^2 dx dy}.$$
A result of the first and third author \cite{Jittered} shows that 
\textit{any} decomposition into $N$ sets of equal measure \textit{always} 
yields a smaller expected squared $\mathcal{L}_2-$discrepancy than
$N$ completely randomly chosen points: even the most primitive Jittered Sampling 
construction is better than Monte Carlo.
However, currently known quantitative bounds \cite{Jittered} do not imply
any effective improvement for $N \lesssim (2d)^{2d}$ points. 
The motivation of our paper is to gain a better understanding of Jittered
Sampling.
A better quantitative control is desirable since 
it could provide a possible way towards improving
bounds on the inverse of the star-discrepancy (see \cite{aisti, dick, hnww}).

\subsection{Result.} 
The purpose of this short note is to initiate the study 
of explicit effective Jittered Sampling constructions by 
obtaining a complete solution for the $d=2, n=2$ case within a natural family 
of domain partitions.
We first consider some natural examples (see Figure \ref{fig:examples1}) of partitions into sets of equal measure and compute their expected squared 
$\mathcal{L}_2-$discrepancy (for details on the computation, see the proof).
\begin{center}
\begin{figure}[h!]
\begin{tikzpicture}[scale=2]
\draw [thick] (0,0) -- (1,0);
\draw [thick] (1,1) -- (1,0);
\draw [thick] (0,1) -- (1,1);
\draw [thick] (0,0) -- (0,1);
\node at (0.5, 0.5) {Monte Carlo};
\node at (0.45, -0.2) {0.0694};
\draw [thick] (2,0) -- (3,0);
\draw [thick] (3,1) -- (3,0);
\draw [thick] (2,1) -- (3,1);
\draw [thick] (2,0) -- (2,1);
\draw [ultra thick] (6,1) -- (7,0);
\node at (2.45, -0.2) {0.0638};
\draw [thick] (4,0) -- (5,0);
\draw [thick] (5,1) -- (5,0);
\draw [thick] (4,1) -- (5,1);
\draw [thick] (4,0) -- (4,1);
\draw [ultra thick] (2,0) -- (3,1);
\node at (4.45, -0.2) {0.0555};
\draw [thick] (6,0) -- (7,0);
\draw [thick] (7,1) -- (7,0);
\draw [thick] (6,1) -- (7,1);
\draw [thick] (6,0) -- (6,1);
\draw [ultra thick] (4,0.5) -- (5,0.5);
\node at (6.45, -0.2) {0.05};
\end{tikzpicture}
\caption{Different subdivisons and their expected squared $\mathcal{L}^2-$discrepancy.} \label{fig:examples1}
\end{figure}
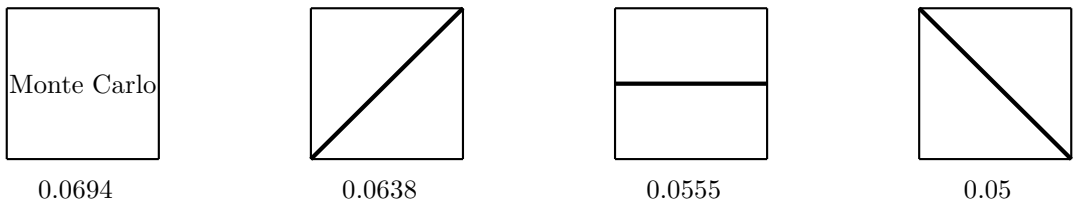
\end{center}
\vspace{-15pt}
These examples suggest that a certain type of symmetry along the $y=x$ diagonal 
seems to be helpful. 
Of course, this still leaves a very large number of shapes that could 
potentially be tested. 
Two natural examples are a quarter disk and a polyhedral domain; see Figure \ref{fig:examples2}.
\begin{center}
\begin{figure}[h!]
\label{sym}
\begin{tikzpicture}[scale=2.7]
\draw[thick] (0,0)--(0,1)--(1,1)--(1,0)--(0,0);
\draw[thick] (0,0.823)--(0.608,0.608)--(0.823,0);
\node at (0,0.823) {$\bullet$}; 
\node at (0.608,0.60) {$\bullet$}; 
\node at (0.823,0) {$\bullet$};  
\node at (0.45, -0.2) {0.0470};
\draw[thick] (-2,0)--(-1,0)--(-1,1)--(-2,1)--(-2,0);
\node at (-1.55, -0.2) {0.0471};
\draw [thick,domain=0:90] plot ({0.8*cos(\x)-2}, {0.8*sin(\x)});
\end{tikzpicture}
\caption{A quarter disk and lines connecting $(0,0.792), (0.63, 0.63)$ and $(0.792, 0)$.} \label{fig:examples2}
\end{figure}
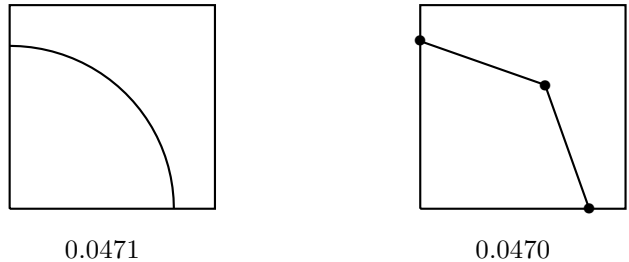
\end{center}
\vspace{-10pt}
We will  now restrict ourselves to the study of partitions of $[0,1]^2 = \Omega \cup \left( [0,1]^2 \setminus \Omega \right),$
where
$$ \Omega = \left\{(x,y) \in [0,1]^2: y \leq g(x) \right\}$$
for some monotonically decreasing function $g$ whose graph $\left\{(x, g(x)): 0 \leq x \leq 1\right\}$ is assumed to be symmetric around the line $y=x$
and splits the unit square into two regions with areas $p$ and $1-p$.
Somewhat to our surprise, it is actually possible to determine a highly 
nonlinear integral equation that any optimal function $g(x)$ has to satisfy. 
The equation is so nonlinear that we would not know of any way to show 
existence of solution, except that it arises as the minimum of a variational
problem for which compactness methods can be used.

\begin{thm}  \label{thm1}
Any optimal monotonically decreasing function $g(x)$ whose
graph is symmetric about $y=x$ satisfies, for $0\leq x\leq g^{-1}(0)$,
\begin{align*}
 &\left(1-2p-4xg(x)\right)\left(1-g(x)\right) +(4p-1)x \left(1-g(x)^2\right) 
 - 4\int_{g(x)}^{g^{-1}(0)}{ (1-y)g^{}(y)dy}  \\
&+ g'(x) \left(  
\left(1-2p-4xg(x)\right)\left(1-x\right) +(4p-1)g(x) \left(1-x^2\right) 
- 4\int_{x}^{g^{-1}(0)}{ (1-y)g(y)  dy} \right) = 0.
\end{align*}
\end{thm}
We use the integral equation to obtain the best numerical approximation to 
the solution $g(x)$ in the space of polynomials of 
degree less than or equal to $10$ using the following procedure.
For a fixed value $\alpha$, the 
integral equation is enforced at $200$ Gauss-Legendre nodes on the
interval $[0,\alpha]$ along with the constraints
$g(0) = \alpha$ and $g(\alpha)=0$. 
The integrals in the equation are computed
using adaptive Gaussian quadrature with error less than $10^{-9}$.
This results in a non-linear least square problem which is solved using 
a Gauss-Newton iterative scheme. 
The solution obtained by this procedure has two issues -- a) the graph
of $g$ is not symmetric and, b) $\int_{0}^{\alpha} g(x) dx \neq p$.
Let $x_{0}$ denote the intersection of $g$ with the line $y = x$.
We use the following symmetrized version of $g$, 
$$
g_{\text{sym}}(x) = \begin{cases}
g(x) &\quad x \leq x_{0} \\
g^{-1}(x) & \quad x>x_{0} \, ,
\end{cases}  
$$
For $p\leq 0.5$, the computed approximation of $g(x)$ was not monotonically 
decreasing on $[0,\delta]$ with $\delta<0.08$. 
In this case, we set $g(x)=y_{\text{max}}$ for $x\leq x_{\text{max}}$, where $y_{\text{max}}$
is the maximum value of $g$ attained at $x=x_{\text{max}}$.
We then vary the parameter $\alpha$ and use bisection
to  ensure the integral of the area under the curve $g(x)$ is $p$ with a tolerance of $10^{-6}$.
The expected $\mathcal{L}^{2}-$ discrepancy of the
sampling procedure arising from the partitions formed by 
the numerical solution is then computed using adaptive Gaussian quadrature
with error less than $10^{-6}$. 
\begin{center}
\begin{figure}[h!]
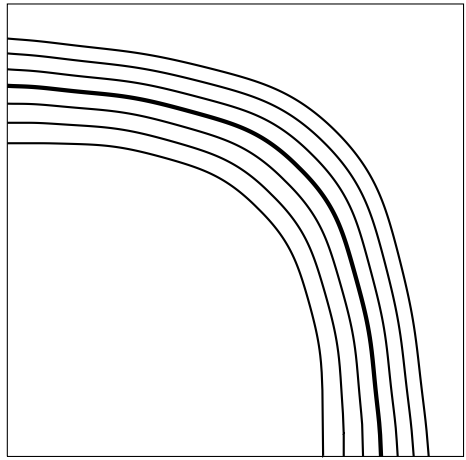

\label{sym}

\caption{Numerically constructed solution: 
for areas 0.423, 0.473, 0.523, 0.573 (bold), 0.623, 0.673, and 0.723. 
The absolute minimum is attained for the bold curve.} \label{fig:cases}
\end{figure}
\end{center}
\vspace{-10pt}
The main implication of the numerical computation is that the optimal 
shape (within the subclass considered) 
does not partition the unit square into two domains of equal area (which may
be an intuitive assumption) and instead opts for 0.573 -- 0.427 split. We believe this
to be quite interesting since it suggests the possibility of optimal partitions in
high dimensions to be fairly asymmetric for a small number of points. Conversely,
as the number of points becomes large, one would assume that most sets in the
partition end up having comparable measure.

\subsection{Open problems.}
It is rather remarkable that the case $d=2, n=2$ 
can be solved essentially completely; we do
not expect any other cases $d=2, n\geq 3$ to be equally accessible. 
The bigger challenge is to understand whether it is possible to construct any 
reasonable rules for partitioning the space that yield quantifiably better 
outcomes than Monte-Carlo. 
It is not even clear how to even pick three or more points from 
$[0,1]^2$. Furthermore, 
as shown above, there is no reason to assume that the optimal distribution will
partition the domain into sets of equal measure.

\section{proof of the theorem}
The proof consists of three parts: first we use integration by parts and some identities to reformulate the variational problem; the second step
uses an Euler-Lagrange type of approach to derive a geometric property that is 
satisfied by the extremal partition and the final part is to use this
geometric property on an infinitesimal scale to derive an integral equation.
\subsection{Preliminaries} 
We will only prove the result for $p=1/2$, 
the general case proceeds in the same manner,  
but with more tedious algebraic computations that are not insightful. 
Let $g$ be a monotonically decreasing function which bounds a minimizing domain
$\Omega$.
Let $p_{1}$ and $p_{2}$ be uniform random variables on $\Omega$ and 
$[0,1]^{2} \setminus \Omega$ respectively, 
and let $\mathcal{B}(q)$ denote a Bernoulli $\left\{0,1\right\}-$variable 
that assumes the value 1 with likelihood $q$. 
We wish to minimize the quantity
$$ \mathbb{E}_{} \int_{[0,1]^2}{ \left| \frac{X(x,y)}{2} - xy\right|^2 dxdy} 
=   
\int_{[0,1]^2}{ \mathbb{E}_{} \left| \frac{X(x,y)}{2} - xy \right|^2 dxdy},$$
where $X(x,y)$ is a random variable given by
$$ X(x,y) := \# \left\{ \left\{p_1,p_2\right\}: p_i \in [0,x] \times [0,y] \right\}.$$
We use a trick from \cite{stein} and simplify the object by exchanging the 
order of integration: let us therefore consider the expectation of the integrand 
for a fixed point $(x,y) \in [0,1]^2$. 
Then $$ \mathbb{E}  \frac{X(x,y)}{2} = x y.$$
This implies 
$$ \mathbb{E}  \left| \frac{X(x,y)}{2} - xy \right|^2  =  
\mathbb{E} \left| \frac{X(x,y)}{2} -   
\mathbb{E}\frac{X(x,y)}{2}  \right|^2 = 
\mbox{var}\left(\frac{X(x,y)}{2} \right) = \frac{1}{4} \mbox{var}(X(x,y)).$$
$X(x,y)$ is the sum of two independent Bernoulli random variables 
$$ 
X(x,y) = \mathcal{B}(2f_{1}) + \mathcal{B}(2f_{2}) \, ,
$$
where
\begin{align*}
f_1(x,y) &= \left|\Omega \cap  \left([0,x] \times [0,y]\right) \right| \\
f_2(x,y) &= \left|\left([0,1]^{2}\setminus\Omega \right) \cap  
\left([0,x] \times [0,y]\right) \right|.
\end{align*}
Since the variance of the sum of two independent random variables is
the sum of their variances, and
$$ \mbox{var}(\mathcal{B}(q)) = q(1-q) \, ,$$
we get
\begin{align*}
4\cdot \mathbb{E}_{} \left| \frac{X(x,y)}{2} - xy \right|^2  
%&= \mbox{var}(\mathcal{B}(  |\Omega_1| \cap | [(0,0), (x,y)] | ) + \mbox{var}( \mathcal{B}(  |\Omega_2| \cap | [(0,0), (x,y)] | )  \\
&= \mbox{var}(\mathcal{B}(  2f_{1})) + \mbox{var}( \mathcal{B}(  2f_{2} ))  \\
&= 2f_1(1-2f_1) + 2f_2(1-2f_2) \\
&= 2f_1 - 4 f_1^2 + 2f_2 - 4f_2^2 \, .
\end{align*}
Furthermore, 
$$ \mathbb{E} \frac{X(x,y)}{2} = xy \implies f_{1}(x,y)+f_{2}(x,y) = xy \, ,$$
and thus
$$ 
4\cdot \mathbb{E}_{} \left| \frac{X(x,y)}{2} - xy \right|^2  
= 
2xy - 4x^2 y^2 + 8f_{1}(x,y)xy - 8f_{1}(x,y)^2 \, .
$$
This means we wish to minimize
\begin{align*}
\mathbb{E}_{} \int_{[0,1]^2}{ \left| \frac{X(x,y)}{2} 
- xy\right|^2 dxdy}  &= 
\int_{[0,1]^2}{ \mathbb{E}_{} \left| \frac{X(x,y)}{2} - xy \right|^2 dxdy} \\
&= \frac{1}{2}\int_{[0,1]^2}{ xy - 2 x^2 y^2 + 4 f_1(x,y) x y - 4f_1(x,y)^2 dx dy } \\
&= \frac{1}{72} + 2\int_{[0,1]^2}{f_1(x,y) x y - f_1(x,y)^2 dx dy }
\end{align*}
We will henceforth work with this reformulation of the problem.

\subsection{Geometric property of extremizers.}
The extremizing function $g$ has to satisfy a very particular geometric 
condition (see Figure \ref{fig:property}),
which we can use to derive an integral equation. We start by stating and 
proving the geometric property.

\begin{center}
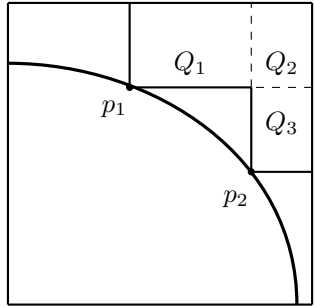
\begin{figure}[h!]
\label{sketch}
\begin{tikzpicture}[scale=4]
\draw [thick] (0,0) -- (1,0);
\draw [thick] (1,1) -- (1,0);
\draw [thick] (0,0) -- (0,1);
\draw [thick] (0,1) -- (1,1);
\draw[very thick] (0,0.8) to [out=0,in=90] (0.95,0);
\filldraw (0.4, 0.72) circle (0.01cm);
\draw [thick] (0.4, 0.72) -- (0.8, 0.72);
\draw [dashed] (1, 0.72) -- (0.8, 0.72);
\filldraw (0.8, 0.44) circle (0.01cm);
\node at (0.35, 0.65) {$p_1$};
\draw [thick] (0.8, 0.44) -- (0.8, 0.72);
\draw [dashed] (0.8, 0.72) -- (0.8, 1);
\draw [thick] (0.4, 0.72) -- (0.4, 1);
\node at (0.75, 0.35) {$p_2$};
\draw [thick] (1, 0.44) -- (0.8, 0.44);
\node at (0.6, 0.8) {$Q_1$};
\node at (0.9, 0.8) {$Q_2$};
\node at (0.9, 0.6) {$Q_3$};
\end{tikzpicture}
\caption{The geometric property satisfied by the extremizer.} \label{fig:property}
\end{figure}
\end{center}
\begin{lemma}
For any two points $p_1,p_2$ on the curve, we always have
$$ \int_{Q_1}{ (2f_{1}(x,y) - xy) dx dy} =  \int_{Q_3}{  ( 2f_{1}(x,y) - xy) dx dy}.$$
\end{lemma}
\begin{proof} The proof proceeds with a variational argument. Let us fix two points $p_1$ and $p_2$ on the curve and try to understand the effect of taking
a tiny $\varepsilon$ unit of area and moving it from $p_1$ to $p_2$. Recall that we are trying to maximize
$$ \int_{[0,1]^2}{ f_{1}(x,y)(f_{1}(x,y)-xy)  dx dy }.$$
If we first remove an $\varepsilon$ unit of area, the change in the integral is to first order
$$  -\int_{Q_1 \cup Q_2}{ f_1(f_1-xy)  dx dy }  + \int_{Q_1 \cup Q_2}{ 
(f_1-\varepsilon)(f_1-\varepsilon-xy)  dx dy } 
\sim \varepsilon\int_{Q_1 \cup Q_2}{ (-2f_1 + xy) dxdy}.$$
We then add this to the point $p_2$ and get a change of order
$$ \int_{Q_2 \cup Q_3}{ (f_1+\varepsilon)(f_1+\varepsilon-xy)  dx dy }  - 
\int_{Q_2 \cup Q_3}{ f_1(f_1-xy)  dx dy }  \sim \varepsilon 
\int_{Q_2 \cup Q_3}{(2f_1 - xy) dx dy }.$$
In the usual manner of these arguments, we conclude that these two quantities have to add up to 0 if we are dealing with a maximizer (otherwise we find something larger, or, if it is negative, we start
moving area from $p_2$ to $p_1$); this means we need
$$ \int_{Q_1 \cup Q_2}{ (-2f_1 + xy) dxdy} + \int_{Q_2 \cup Q_3}{(2f_1 - xy) dx dy } = 0$$
which proves the lemma.
\end{proof}

\subsection{Derivation of the integral equation.} 
Since $g$ is a monotonically decreasing function, a classical consequence is 
that $g$ is differentiable almost everywhere -- we will now use the geometric Lemma as $p_1 \rightarrow p_2$ (assuming that
$g$ is differentiable in $p_2$). Then $Q_1$ and $Q_3$ degenerate to very thin 
rectangles.
We need to understand how
$Q_1$ and $Q_3$ degenerate.
Let $\ell_{1}$ and $\ell_{3}$ denote the limiting lines corresponding
to the degenerate rectangles $Q_{1}$ and $Q_{3}$ as $p_1 \to p_2 = (x,g(x))$. 
The degenerate sides of the rectangles $Q_{1}$ and $Q_{3}$ 
then scale as $|\cos{\alpha}|$ and $|\sin{\alpha}|$ respectively, where
$$ \tan{(\alpha)} = g'(x) \, .$$
Using the identities 
$$ \cos{\arctan{\theta}} = \frac{1}{\sqrt{1+\theta^2}} 
\quad \mbox{and} \quad \sin{\arctan{\theta}} = \frac{\theta}{\sqrt{1+\theta^2}}$$
we deduce that the ratio of the degenerate sides of $Q_{1}$ and $Q_{3}$ 
in the limit is $1:|g'(x)|$.
Using the Lemma, and taking the limit $p_1\to p_2$, we get
\begin{align*}
\int_{\ell_{1}} (2f_{1}(x,y)-xy) dx dy &= -g'(x) \int_{\ell_{3}} 
(2f_{1}(x,y)-xy) dx dy
\end{align*}
The integral along $\ell_{1}$ is given by
$$\int_{\ell_1}{f_1(x,y)dxdy} =  \int_{g(x)}^{1}{\left(x g(x) + \int_{g(x)}^{y}{g^{-1}(z) dz}\right) dy} = x g(x) (1-g(x)) + \int_{g(x)}^{1}{\left( \int_{g(x)}^{y}{g^{-1}(z) dz}\right) dy} .$$
Using $g = g^{-1}$ and a change of variables, this simplifies further to
$$
 x g(x) (1-g(x)) + \int_{g(x)}^{1}{\left( \int_{g(x)}^{y}{g^{}(z) dz}\right) dy} =  x g(x) (1-g(x))  + \int_{g(x)}^{1}{ (1-y)g^{}(y)dy}.
$$
It is easy to see that along a line
$$ \int_{\ell_1}{\frac{xy}{2}dxdy} = \int_{g(x)}^{1}{x y dy} = 
x \frac{1-g(x)^2}{4}.$$
By the same token we have
\begin{align*}
\int_{\ell_3}{f_1(x,y)dxdy} =  \int_{x}^{1}{\left(x g(x) + \int_{x}^{y}{g^{}(z) dz}\right) dy} &= x g(x) (1-x) + \int_{x}^{1}{\left( \int_{x}^{y}{g(z) dz}\right) dy} \\
&= x g(x) (1-x) + \int_{x}^{1}{ (1-y)g(y)  dy}
\end{align*}
and
$$ \int_{\ell_3}{\frac{x y}{2} dxdy} = \int_{x}^{1}{z g(x) dz} = g(x)\frac{1-x^2}{4}.$$
Altogether, we have
\begin{align*}
 & x \left(g(x)-\frac{1}{4}\right) - 3\frac{xg(x)^2}{4}  
 + \int_{g(x)}^{1}{ (1-y)g^{}(y)dy}  \\
&+ g'(x) \left( g(x) \left(x-\frac{1}{4}\right) - \frac{3x^2g(x)}{4}  + \int_{x}^{1}{ (1-y)g(y)  dy} \right) = 0,
\end{align*}
which is the equation in Theorem \ref{thm1} for the case $p=1/2$.
The general case follows in a similar manner.

\end{document}